\documentclass[oneside,reqno,english]{amsart}
\usepackage[T1]{fontenc}
\usepackage[utf8]{inputenc}
\usepackage{xcolor}
\usepackage{babel}
\usepackage{prettyref}
\usepackage{amsthm}
\usepackage{amssymb}
\usepackage[pdfusetitle,
 bookmarks=true,bookmarksnumbered=false,bookmarksopen=false,
 breaklinks=false,pdfborder={0 0 0},pdfborderstyle={},backref=false,colorlinks=false]
 {hyperref}
\hypersetup{
 colorlinks=true,citecolor=blue,linkcolor=blue,linktocpage=true}

\makeatletter
\numberwithin{equation}{section}
\numberwithin{figure}{section}

\usepackage{prettyref}

\newrefformat{cor}{Corollary~\ref{#1}}
\newrefformat{subsec}{Section~\ref{#1}}
\newrefformat{lem}{Lemma~\ref{#1}}
\newrefformat{thm}{Theorem~\ref{#1}}
\newrefformat{sec}{Section~\ref{#1}}
\newrefformat{chap}{Chapter~\ref{#1}}
\newrefformat{prop}{Proposition~\ref{#1}}
\newrefformat{exa}{Example~\ref{#1}}
\newrefformat{tab}{Table~\ref{#1}}
\newrefformat{rem}{Remark~\ref{#1}}
\newrefformat{def}{Definition~\ref{#1}}
\newrefformat{fig}{Figure~\ref{#1}}
\newrefformat{claim}{Claim~\ref{#1}}
\newrefformat{assu}{Assumption~\ref{#1}}

\makeatother

\theoremstyle{plain}
\newtheorem{thm}{\protect\theoremname}[section]
\newtheorem{prop}[thm]{\protect\propositionname}
\newtheorem{cor}[thm]{\protect\corollaryname}
\providecommand{\corollaryname}{Corollary}
\providecommand{\propositionname}{Proposition}
\providecommand{\theoremname}{Theorem}

\begin{document}
\title{A finite-Point Schwarz-Pick Inequality}
\begin{abstract}
We prove a finite-point version of the Schwarz-Pick inequality for
holomorphic self-maps of the disk. The estimate compares Bergman Gram
matrices built from several points and their images, and the sharp
constant depends on both the number of points and, for finite Blaschke
products, the degree of the map. These matrices also arise from the
Szegő-kernel metric on the symmetrized polydisc, where the estimate
becomes a sharp Schwarz-Pick inequality for the induced holomorphic
maps.
\end{abstract}

\author{James Tian}
\address{Mathematical Reviews, 535 W. William St, Suite 210, Ann Arbor, MI
48103, USA}
\email{james.ftian@gmail.com}

\maketitle
\tableofcontents{}

\section{Introduction}\label{sec:1}

The Schwarz-Pick lemma is usually read one point at a time. If $f\colon\mathbb{D}\longrightarrow\mathbb{D}$
is holomorphic, then 
\begin{equation}
\frac{\left|f'(z)\right|}{1-\left|f(z)\right|^{2}}\leq\frac{1}{1-\left|z\right|^{2}}.\label{eq:1-1}
\end{equation}
Geometrically, this says that a holomorphic self-map of the disk contracts
the Poincaré metric. If several points $z_{1},\ldots,z_{n}$ are given,
one can of course apply the inequality at each point separately. That
gives $n$ scalar inequalities, but it does not say how the derivatives
at the different points fit together.

There is a well developed multipoint form of Schwarz-Pick theory.
One approach uses hyperbolic difference quotients and their iterates.
It leads naturally to the Schur algorithm, Nevanlinna-Pick interpolation,
and refinements in which information at several points is fed successively
into the estimate \cite{MR2072742,MR2572646,MR3058518}. Schwarz-Pick
inequalities have also been studied through reproducing kernel spaces
and operator theory \cite{MR2304058,MR2419482,MR4760925}. The latter
point of view has recently been used to connect multipoint inequalities
with model spaces and invariant derivatives \cite{MR4760925}.

We take a different finite-point view. We keep the ordinary derivatives
\[
f'\left(z_{1}\right),\ldots,f'\left(z_{n}\right)
\]
and ask for a single inequality that sees them all at once. The Bergman
kernel gives a natural way to do this. For $Z=\left(z_{1},\ldots,z_{n}\right)$,
let $G_{Z}$ and $H_{f,Z}$ be the matrices defined in \prettyref{eq:2-1}
and \prettyref{eq:2-2}. Our main result is 
\[
H_{f,Z}\leq nG_{Z}.
\]
Here and throughout, the inequality is in the Loewner order on Hermitian
matrices. ($A\leq B$ means that $B-A$ is positive semidefinite.)
The constant $n$ is sharp.

The matrix order matters. The content is not in the diagonal estimates,
but in the way the off-diagonal entries couple the different points.
The full inequality controls every linear combination of the corresponding
Bergman kernels, and therefore contains information that is not obtained
by applying the scalar Schwarz-Pick lemma $n$ times.

There is a second feature when the map is a finite Blaschke product.
If $B$ has degree $d$, then the estimate improves to 
\[
H_{B,Z}\leq\min\left\{ n,d\right\} G_{Z},
\]
and this constant is again sharp. 

The proof is built from a few Hilbert space ideas. For a finite Blaschke
product $B$, the Bergman kernel factors through the model space 
\[
K_{B}=H^{2}\ominus BH^{2}.
\]
The canonical conjugation on $K_{B}$ recovers $B'$ by a tensor contraction.
For the chosen points, only the space 
\[
span\left\{ q^{B}_{z_{1}},\ldots,q^{B}_{z_{n}}\right\} 
\]
is involved. Its dimension is bounded by both $n$ and $\dim K_{B}=d$,
and the desired constant comes from this simple dimension count. The
general case then follows by approximation with finite Blaschke products.

For one point the same argument reduces to Cauchy-Schwarz and gives
the classical Schwarz-Pick lemma. In this sense the finite-point inequality
is the same Hilbert space approach after several evaluation vectors
are allowed to interact.

There is also an operator-theoretic way to read the estimate. Suppose
for simplicity that $z_{1},\ldots,z_{n}$ are distinct. On the span
of the Bergman kernels at the chosen points, consider the map 
\[
T_{f,Z}k_{z_{i}}=\overline{f'\left(z_{i}\right)}k_{f\left(z_{i}\right)}.
\]
Then \prettyref{eq:2-3} says that 
\[
\left\Vert T_{f,Z}\right\Vert \leq\sqrt{n},
\]
and for a finite Blaschke product of degree $d$ the bound improves
to $\sqrt{\min\left\{ n,d\right\} }$. So this gives a sharp finite-dimensional
estimate for the derivative weighted composition map determined by
the chosen points.

The inequality also gives a necessary condition for interpolation
with prescribed values and first derivatives. If $f\left(z_{i}\right)=w_{i}$
and $f'\left(z_{i}\right)=a_{i}$ for a holomorphic self-map of the
disk, then 
\[
\left[\frac{a_{i}\overline{a_{j}}}{\left(1-w_{i}\overline{w_{j}}\right)^{2}}\right]^{n}_{i,j=1}\leq n\left[\frac{1}{\left(1-z_{i}\overline{z_{j}}\right)^{2}}\right]^{n}_{i,j=1}.
\]
This does not replace the usual Pick interpolation criteria, but it
gives a direct simultaneous restriction on the first-order data \cite{MR208383,MR2572646,MR3058518}.

Lastly, the finite-point estimate also has a geometric interpretation.
In \prettyref{sec:3} we show that the matrices above arise from the
Szegő kernel metric on the symmetrized polydisc. The estimate then
becomes a Schwarz-Pick inequality for the self-maps induced by holomorphic
maps of the disk. For proper self-maps, the classification by finite
Blaschke products gives a sharp bound in terms of the degree of the
map.

\section{The finite-point estimate}\label{sec:2}

We begin directly with the finite-point question. Let $A^{2}$ denote
the Bergman space on $\mathbb{D}$, with reproducing kernel 
\begin{equation}
k_{z}(w)=\frac{1}{\left(1-\overline{z}w\right)^{2}}.\label{eq:2-1a}
\end{equation}
For a holomorphic self-map $f$ of the disk and points $z_{1},\ldots,z_{n}\in\mathbb{D}$,
consider the two matrices 
\begin{equation}
G_{Z}=\left[\frac{1}{\left(1-z_{i}\overline{z_{j}}\right)^{2}}\right]^{n}_{i,j=1}\label{eq:2-1}
\end{equation}
and 
\begin{equation}
H_{f,Z}=\left[\frac{f'\left(z_{i}\right)\overline{f'\left(z_{j}\right)}}{\left(1-f\left(z_{i}\right)\overline{f\left(z_{j}\right)}\right)^{2}}\right]^{n}_{i,j=1}.\label{eq:2-2}
\end{equation}
The first is the Gram matrix of the Bergman kernels at the points
$z_{1},\ldots,z_{n}$. When $n=1$, the inequality $H_{f,Z}\leq G_{Z}$
is the usual Schwarz-Pick lemma \prettyref{eq:1-1}. For several points
there is a different bound.
\begin{thm}
\label{thm:2-1} Let $f$ be a holomorphic self-map of $\mathbb{D}$
and let $z_{1},\ldots,z_{n}\in\mathbb{D}$. Then 
\begin{equation}
H_{f,Z}\leq nG_{Z}.\label{eq:2-3}
\end{equation}
If $f=B$ is a finite Blaschke product of degree $d$, then 
\begin{equation}
H_{B,Z}\leq\min\left\{ n,d\right\} G_{Z}.\label{eq:2-4}
\end{equation}
Both constants are sharp.
\end{thm}

\begin{proof}
We first consider a finite Blaschke product $B$ of degree $d$. Let
\[
K_{B}=H^{2}\ominus BH^{2}
\]
be its model space. Then $\dim K_{B}=d$, and the reproducing kernel
of $K_{B}$ is 
\[
q^{B}_{z}(w)=\frac{1-\overline{B(z)}B(w)}{1-\overline{z}w}.
\]

The useful observation is a direct factorization of the Bergman kernel.
For $z,w\in\mathbb{D}$, 
\[
\left\langle k_{z},k_{w}\right\rangle =\left\langle q^{B}_{z},q^{B}_{w}\right\rangle ^{2}\left\langle k_{B(z)},k_{B(w)}\right\rangle .
\]
Indeed, the factors involving $B$ cancel, leaving 
\[
\frac{1}{\left(1-z\overline{w}\right)^{2}}.
\]
It follows that there is an isometry 
\begin{equation}
\Gamma\colon A^{2}\longrightarrow K_{B}\otimes K_{B}\otimes A^{2}\label{eq:2-5}
\end{equation}
such that 
\begin{equation}
\Gamma k_{z}=q^{B}_{z}\otimes q^{B}_{z}\otimes k_{B(z)}.\label{eq:2-6}
\end{equation}
One can see this directly on finite linear combinations of reproducing
kernels. The identity of Gram matrices shows that the rule in \prettyref{eq:2-6}
preserves norms, and the span of the Bergman kernels is dense in $A^{2}$.

We now look for the derivative of $B$ inside the first two tensor
factors. Let $C_{B}$ be the canonical conjugation on $K_{B}$, given
on the boundary by 
\[
\left(C_{B}g\right)(\zeta)=B(\zeta)\overline{\zeta g(\zeta)},\qquad\left|\zeta\right|=1.
\]
It is an antilinear isometric involution on $K_{B}$. Its action on
the reproducing kernels is 
\begin{equation}
\left(C_{B}q^{B}_{z}\right)(w)=\frac{B(w)-B(z)}{w-z},\label{eq:2-7}
\end{equation}
where the value at $w=z$ is obtained by removing the singularity.
In particular, 
\[
\left(C_{B}q^{B}_{z}\right)(z)=B'(z).
\]

For the fixed points $z_{1},\ldots,z_{n}$, let 
\[
R_{Z}=span\left\{ q^{B}_{z_{1}},\ldots,q^{B}_{z_{n}}\right\} \subseteq K_{B}
\]
and write 
\[
r=\dim R_{Z}.
\]
Since $R_{Z}$ is spanned by $n$ vectors and lies in the $d$-dimensional
space $K_{B}$, 
\begin{equation}
r\leq\min\left\{ n,d\right\} .\label{eq:2-8}
\end{equation}

Choose an orthonormal basis $e_{1},\ldots,e_{d}$ of $K_{B}$ such
that $e_{1},\ldots,e_{r}$ is an orthonormal basis of $R_{Z}$, and
set 
\begin{equation}
\Omega_{B}=\sum^{d}_{j=1}e_{j}\otimes C_{B}e_{j}.\label{eq:2-9}
\end{equation}
The vectors $C_{B}e_{1},\ldots,C_{B}e_{d}$ are again orthonormal,
so 
\[
\left\Vert \Omega_{B}\right\Vert ^{2}=d.
\]

For $x,y\in K_{B}$, 
\[
\left\langle \Omega_{B},x\otimes y\right\rangle =\left\langle C_{B}x,y\right\rangle .
\]
Taking $x=y=q^{B}_{z}$ and using \prettyref{eq:2-7} gives 
\begin{equation}
\left\langle \Omega_{B},q^{B}_{z}\otimes q^{B}_{z}\right\rangle =\overline{B'(z)}.\label{eq:2-10}
\end{equation}
So the derivative is obtained by contracting the first two tensor
factors in \prettyref{eq:2-6} against $\Omega_{B}$.

Normalize this contraction by defining $L\colon K_{B}\otimes K_{B}\otimes A^{2}\longrightarrow A^{2}$
on elementary tensors by 
\[
L\left(x\otimes y\otimes h\right)=\frac{1}{\sqrt{d}}\left\langle \Omega_{B},x\otimes y\right\rangle h.
\]
We only need $L$ on $R_{Z}\otimes K_{B}\otimes A^{2}$. On the first
two factors, the representing vector for this restricted functional
is 
\[
\frac{1}{\sqrt{d}}\sum^{r}_{j=1}e_{j}\otimes C_{B}e_{j}.
\]
Its squared norm is $r/d$. Hence 
\begin{equation}
\left\Vert L\big|_{R_{Z}\otimes K_{B}\otimes A^{2}}\right\Vert ^{2}=\frac{r}{d}.\label{eq:2-11}
\end{equation}

Now take 
\[
u=\sum^{n}_{i=1}c_{i}k_{z_{i}}.
\]
By \prettyref{eq:2-6}, 
\[
\Gamma u=\sum^{n}_{i=1}c_{i}q^{B}_{z_{i}}\otimes q^{B}_{z_{i}}\otimes k_{B(z_{i})},
\]
so $\Gamma u$ belongs to $R_{Z}\otimes K_{B}\otimes A^{2}$. Since
$\Gamma$ is an isometry, \prettyref{eq:2-11} gives 
\begin{equation}
\left\Vert L\Gamma u\right\Vert ^{2}\leq\frac{r}{d}\left\Vert u\right\Vert ^{2}.\label{eq:2-12}
\end{equation}
On the other hand, \prettyref{eq:2-10} gives 
\[
L\Gamma u=\frac{1}{\sqrt{d}}\sum^{n}_{i=1}c_{i}\overline{B'\left(z_{i}\right)}k_{B(z_{i})}.
\]
Therefore 
\[
d\left\Vert L\Gamma u\right\Vert ^{2}=\sum^{n}_{i,j=1}\overline{c_{i}}c_{j}\frac{B'\left(z_{i}\right)\overline{B'\left(z_{j}\right)}}{\left(1-B\left(z_{i}\right)\overline{B\left(z_{j}\right)}\right)^{2}}=c^{*}H_{B,Z}c.
\]
Similarly, 
\[
\left\Vert u\right\Vert ^{2}=\sum^{n}_{i,j=1}\overline{c_{i}}c_{j}\frac{1}{\left(1-z_{i}\overline{z_{j}}\right)^{2}}=c^{*}G_{Z}c.
\]
Multiplying \prettyref{eq:2-12} by $d$ therefore gives 
\[
c^{*}H_{B,Z}c\leq rc^{*}G_{Z}c.
\]
This holds for every $c\in\mathbb{C}^{n}$, so in fact 
\[
H_{B,Z}\leq rG_{Z}.
\]
Together with \prettyref{eq:2-8}, this proves \prettyref{eq:2-4}.

We pass to an arbitrary holomorphic self-map $f\colon\mathbb{D}\longrightarrow\mathbb{D}$.
By the Schur algorithm, there are finite Blaschke products $B_{m}$
converging to $f$ locally uniformly. Cauchy's formula then gives
$B'_{m}\longrightarrow f'$ locally uniformly as well. For each $m$,
\[
H_{B_{m},Z}\leq\min\left\{ n,\deg B_{m}\right\} G_{Z}\leq nG_{Z}.
\]
For the fixed finite set $Z$, the matrices $H_{B_{m},Z}$ converge
entry by entry to $H_{f,Z}$. Passing to the limit gives \prettyref{eq:2-3}.

It remains to see that the constant in \prettyref{eq:2-4} cannot
be improved. Take $B(z)=z^{d}$. First suppose $n\leq d$. Choose
distinct $d$th roots of unity $\omega_{1},\ldots,\omega_{n}$ and
let $z_{j}=r\omega_{j}$, $0<r<1$. Then all the points lie in the
same fiber of $B$.

For $i=j$, 
\[
\left(G_{Z}\right)_{ii}=\frac{1}{\left(1-r^{2}\right)^{2}},
\]
and for $i\neq j$ the entries of $G_{Z}$ remain bounded as $r$
tends to $1$. Thus 
\begin{equation}
\left(1-r^{2}\right)^{2}G_{Z}\longrightarrow I.\label{eq:2-13}
\end{equation}

Since $B'\left(z_{j}\right)=dr^{d-1}\overline{\omega_{j}}$ and $B\left(z_{j}\right)=r^{d}$,
we have 
\[
H_{B,Z}=\frac{d^{2}r^{2d-2}}{\left(1-r^{2d}\right)^{2}}vv^{*},\qquad v=\begin{pmatrix}\overline{\omega_{1}}\\
\vdots\\
\overline{\omega_{n}}
\end{pmatrix}.
\]
Using $1-r^{2d}=\left(1-r^{2}\right)\left(1+r^{2}+\cdots+r^{2d-2}\right)$,
we get 
\begin{equation}
\left(1-r^{2}\right)^{2}H_{B,Z}\longrightarrow vv^{*}.\label{eq:2-14}
\end{equation}
The matrix $vv^{*}$ has one nonzero eigenvalue, namely $\left\Vert v\right\Vert ^{2}=n$.
Thus any constant $C$ for which 
\[
H_{B,Z}\leq CG_{Z}
\]
holds uniformly must satisfy $C\geq n$.

If $n\geq d$, take a full $d$-point fiber and, when $n>d$, add
any further points in the disk. The same argument applied to the principal
$d\times d$ block gives $C\geq d$. Hence the best universal constant
is at least $\min\left\{ n,d\right\} $, and the upper bound has already
been proved. 
\end{proof}

There is a useful way to compare the proof with the ordinary Schwarz-Pick
lemma. For one point, the finite-dimensional estimate above reduces
to Cauchy-Schwarz. Indeed, 
\[
B'(z)=\left\langle q^{B}_{z},C_{B}q^{B}_{z}\right\rangle ,
\]
and therefore 
\[
\left|B'(z)\right|\leq\left\Vert q^{B}_{z}\right\Vert ^{2}=\frac{1-\left|B(z)\right|^{2}}{1-\left|z\right|^{2}}.
\]
For several points, the same model space is still present. What changes
is that the kernels $q^{B}_{z_{1}},\ldots,q^{B}_{z_{n}}$ can occupy
an $r$-dimensional part of $K_{B}$. The normalized contraction sees
only that part, and its squared norm is $r/d$. Since $r\leq n$ and
$r\leq d$, the two numbers in $\min\left\{ n,d\right\} $ enter for
the same reason.

\section{The symmetrized polydisc}\label{sec:3}

Here it is convenient to use the opposite matrix convention for Hermitian
forms. This replaces $G_{Z}$ and $H_{f,Z}$ by their entrywise conjugates,
which does not change any Loewner-order inequality, so we keep the
same notation.

For $Z=\left(z_{1},\ldots,z_{n}\right)\in\mathbb{D}^{n}$, let 
\[
\pi_{n}\left(Z\right)=\left(\sigma_{1}\left(Z\right),\ldots,\sigma_{n}\left(Z\right)\right),
\]
where 
\[
\sigma_{k}\left(Z\right)=\sum_{1\leq i_{1}<\cdots<i_{k}\leq n}z_{i_{1}}\cdots z_{i_{k}}.
\]
The symmetrized polydisc is 
\[
\mathbb{G}_{n}=\pi_{n}\left(\mathbb{D}^{n}\right).
\]
The map $\pi_{n}$ forgets the ordering of the points and remembers
only the coefficients of the monic polynomial having $z_{1},\ldots,z_{n}$
as its roots.

Let $S_{n}$ denote the Szegő kernel of $\mathbb{G}_{n}$. Its pullback
under $\pi_{n}$ has the particularly simple form \cite{MR3043017}
\begin{equation}
S_{n}\left(\pi_{n}\left(Z\right),\pi_{n}\left(W\right)\right)=\prod^{n}_{i,j=1}\frac{1}{1-z_{i}\overline{w_{j}}}.\label{eq:3-1}
\end{equation}
We consider the Hermitian metric $g_{n}$ obtained from the logarithmic
Hessian of this kernel. In local coordinates $u=\left(u_{1},\ldots,u_{n}\right)$
on $\mathbb{G}_{n}$, its matrix is 
\begin{equation}
\left[g_{n}\left(u\right)\right]_{\alpha\beta}=\frac{\partial^{2}}{\partial u_{\beta}\partial\overline{u_{\alpha}}}\log S_{n}\left(u,u\right).\label{eq:3-2}
\end{equation}
The usual reproducing kernel argument shows that this is positive
definite. We will only need its pullback to root coordinates.

Every holomorphic self-map $f$ of the disk induces a holomorphic
self-map $F_{f}$ of $\mathbb{G}_{n}$ by 
\begin{equation}
F_{f}\left(\pi_{n}\left(z_{1},\ldots,z_{n}\right)\right)=\pi_{n}\left(f\left(z_{1}\right),\ldots,f\left(z_{n}\right)\right).\label{eq:3-3}
\end{equation}
This map is well defined and holomorphic \cite{MR2135687}. The next
calculation identifies the two matrices from \prettyref{thm:2-1}.
\begin{prop}
\label{prop:3-1} Let $Z=\left(z_{1},\ldots,z_{n}\right)\in\mathbb{D}^{n}$.
The matrix of the pulled-back metric $\pi^{*}_{n}g_{n}$ in the coordinates
$z_{1},\ldots,z_{n}$ is $G_{Z}$. If $f$ is a holomorphic self-map
of $\mathbb{D}$, then the matrix of $\pi^{*}_{n}\left(F^{*}_{f}g_{n}\right)$
is $H_{f,Z}$. 
\end{prop}

\begin{proof}
Setting $W=Z$ in \prettyref{eq:3-1} gives 
\[
\log S_{n}\left(\pi_{n}\left(Z\right),\pi_{n}\left(Z\right)\right)=-\sum^{n}_{i,j=1}\log\left(1-z_{i}\overline{z_{j}}\right).
\]
Therefore 
\[
\frac{\partial^{2}}{\partial z_{j}\partial\overline{z_{i}}}\log S_{n}\left(\pi_{n}\left(Z\right),\pi_{n}\left(Z\right)\right)=\frac{1}{\left(1-\overline{z_{i}}z_{j}\right)^{2}}.
\]
This is the matrix $G_{Z}$.

For the second assertion, \prettyref{eq:3-3} gives 
\[
F_{f}\circ\pi_{n}=\pi_{n}\circ f^{\times n},
\]
where 
\[
f^{\times n}\left(z_{1},\ldots,z_{n}\right)=\left(f\left(z_{1}\right),\ldots,f\left(z_{n}\right)\right).
\]
Hence 
\[
\pi^{*}_{n}\left(F^{*}_{f}g_{n}\right)=\left(f^{\times n}\right)^{*}\left(\pi^{*}_{n}g_{n}\right).
\]
Using the first part, the matrix of the right side is 
\[
\left[\frac{\overline{f'\left(z_{i}\right)}f'\left(z_{j}\right)}{\left(1-\overline{f\left(z_{i}\right)}f\left(z_{j}\right)\right)^{2}}\right]^{n}_{i,j=1},
\]
which is $H_{f,Z}$. 
\end{proof}

The finite-point inequality now becomes a Schwarz-Pick estimate for
the metric $g_{n}$.
\begin{cor}
\label{cor:3-2} Let $f$ be a holomorphic self-map of $\mathbb{D}$,
and let $F_{f}$ be the induced self-map of $\mathbb{G}_{n}$. Then
\begin{equation}
F^{*}_{f}g_{n}\leq ng_{n}.\label{eq:3-4}
\end{equation}
If $B$ is a finite Blaschke product of degree $d$, then 
\begin{equation}
F^{*}_{B}g_{n}\leq\min\left\{ n,d\right\} g_{n}.\label{eq:3-5}
\end{equation}
The constants are sharp. 
\end{cor}

\begin{proof}
Suppose first that the coordinates of $Z$ are distinct. Then $\pi_{n}$
is locally biholomorphic at $Z$. By \prettyref{prop:3-1}, the pullbacks
of $g_{n}$ and $F^{*}_{f}g_{n}$ have matrices $G_{Z}$ and $H_{f,Z}$.
The two assertions therefore follow directly from \prettyref{thm:2-1}.
Points with distinct coordinates are dense, so the inequalities extend
by continuity across the collision set.

The sharpness examples in the proof of \prettyref{thm:2-1} use distinct
points. Since $\pi_{n}$ is locally biholomorphic at those points,
the same examples show that the constants in \prettyref{eq:3-4} and
\prettyref{eq:3-5} cannot be decreased. 
\end{proof}

The finite Blaschke case has an intrinsic interpretation on $\mathbb{G}_{n}$.
Edigarian and Zwonek proved that every proper holomorphic self-map
of $\mathbb{G}_{n}$ is $F_{B}$ for some finite Blaschke product
$B$ \cite{MR2135687}. The degree of $B$ can also be read from the
degree of the induced map.

Suppose that $B$ has degree $d$. Choose a generic point 
\[
u=\pi_{n}\left(w_{1},\ldots,w_{n}\right)
\]
with distinct coordinates $w_{1},\ldots,w_{n}$, none of which is
a critical value of $B$. Each $w_{j}$ has $d$ preimages under $B$.
A preimage of $u$ under $F_{B}$ is obtained by choosing one preimage
of each $w_{j}$ and then forgetting the order. Since the fibers over
the distinct points $w_{j}$ are disjoint, this gives $d^{n}$ different
preimages. Consequently, 
\begin{equation}
\deg F_{B}=d^{n}.\label{eq:3-6}
\end{equation}

Combining this observation with \prettyref{cor:3-2} gives a form
of the estimate that refers only to the geometry of the symmetrized
polydisc.
\begin{thm}
\label{thm:3-3} Let $F:\mathbb{G}_{n}\longrightarrow\mathbb{G}_{n}$
be a proper holomorphic map. Then 
\begin{equation}
F^{*}g_{n}\leq\min\left\{ n,\left(\deg F\right)^{1/n}\right\} g_{n}.\label{eq:3-7}
\end{equation}
The constant is sharp. 
\end{thm}

\begin{proof}
By the classification of proper self-maps of the symmetrized polydisc
\cite{MR2135687}, there is a finite Blaschke product $B$ of some
degree $d$ such that $F=F_{B}$. By \prettyref{eq:3-6}, 
\[
d=\left(\deg F\right)^{1/n}.
\]
The estimate now follows from \prettyref{eq:3-5}.

For every $d\geq1$, the map induced by $B\left(z\right)=z^{d}$ has
degree $d^{n}$. The sharpness part of \prettyref{cor:3-2} shows
that the coefficient $\min\left\{ n,d\right\} $ cannot be replaced
by a smaller universal coefficient for proper maps of this degree. 
\end{proof}

In particular, if $F$ is an automorphism of $\mathbb{G}_{n}$, then
\prettyref{thm:3-3} gives 
\[
F^{*}g_{n}\leq g_{n}.
\]
Applying the same inequality to $F^{-1}$ gives the reverse inequality.
Hence every automorphism of $\mathbb{G}_{n}$ is an isometry of $g_{n}$.

\bibliographystyle{amsalpha}
\bibliography{ref}

\end{document}